\documentclass[a4paper,12pt,intlimits,oneside]{amsart}
\usepackage{amsmath}
\usepackage{amsthm}
\usepackage{latexsym}
\usepackage{amssymb}
\usepackage{xcolor}
\usepackage{graphicx}
\usepackage{mathrsfs}
\usepackage{tikz}
\usepackage{xcolor}
\usepackage[colorlinks=true]{hyperref}
\numberwithin{figure}{section}
\def\R{{\mathbb R}}
\def\C{{\mathbb C}}

\def\T{{\mathbb T}}
\def\Z{{\mathbb Z}}

\def\N{{\mathbb N}}
\def\Q{{\mathbb Q}}
\def\phe{\varphi}
\def\s{\vskip 0.25cm\noindent}
\def\e{\varepsilon}
\def\pa{\partial}
\def\build#1_#2^#3{\mathrel{
\mathop{\kern 0pt#1}\limits_{#2}^{#3}}}
\def\td_#1,#2{\mathrel{\mathop{\build\longrightarrow_{#1\rightarrow #2}^{}}}}
\newtheorem{theorem}{Theorem}
\newtheorem{corollary}{Corollary}
\newtheorem{proposition}{Proposition}
\newtheorem{lemma}{Lemma}
\newtheorem{remark}{Remark}

\begin{document}
\title[Generic colourful tori]{Generic colourful tori and inverse spectral transform for Hankel operators}
\author{Patrick G\'erard}
\address{Universit\'e Paris-Sud XI, Laboratoire de Math\'ematiques
d'Orsay, CNRS, UMR 8628, et Institut Universitaire de France} \email{{\tt Patrick.Gerard@math.u-psud.fr}}
\author[S. Grellier]{Sandrine Grellier}
\address{F\'ed\'eration Denis Poisson, MAPMO-UMR 6628,
D\'epartement de Math\'ematiques, Universit\'e d'Orleans, 45067
Orl\'eans Cedex 2, France} \email{{\tt
Sandrine.Grellier@univ-orleans.fr}}


\date{December 4, 2017}

\begin{abstract} {This paper explores the regularity properties of an inverse spectral transform for Hilbert--Schmidt Hankel operators on the unit disc. This spectral transform plays the role of action-angles variables for an integrable infinite dimensional Hamiltonian system -- the cubic Szeg\H{o} equation. We investigate the regularity of functions on the  tori supporting the dynamics of this system, in connection with some wave turbulence phenomenon, discovered in a previous work and due to relative small gaps between the actions.  We revisit this phenomenon by proving that generic smooth functions and a $G_\delta$ dense set of irregular functions do coexist on the same torus. On the other hand, we establish some uniform analytic regularity for tori corresponding to rapidly decreasing actions  which satisfy some specific property ruling out the phenomenon of small gaps. }
\end{abstract}


\thanks {{P. G\'erard is  grateful to G. Chenevier for valuable discussions about elliptic functions. The first author is supported by ANR ANAE 13-BS01-0010-03.}}
\maketitle

\section{Introduction}
\subsection{The cubic Szeg\H{o} equation} This paper explores the properties of  some inverse spectral transformation related to an integrable infinite dimensional Hamiltonian system. Introduced in \cite{GG1}, the cubic Szeg\H{o} equation reads
\begin{equation}\label{szego}
i\pa_tu=\Pi (|u|^2u)\ ,
\end{equation}
where $u=u(t,x)$ is a function defined for $(t,x)\in \R \times \T$, $\T :=\R/2\pi \Z$, such that, for every $t\in \R$, $u(t,.)$ belongs to the Hardy space $L^2_+(\T )$ of $L^2$ functions $v$ on $\T $ with only nonnegative Fourier modes,
$$\forall n<0\ ,\ \hat v(n)=0\ .$$
Here 
$$\hat v(n)=\int_0^{2\pi}v(x)\,{\rm e}^{-inx}\, \frac{dx}{2\pi}\ ,\ n\in \Z$$
denotes the Fourier coefficient of $v\in L^2(\T )$, and $\Pi $ denotes the orthogonal projector from $L^2(\T )$ onto $L^2_+(\T )$,
$$\Pi \left (\sum_{n\in \Z}c_n\,{\rm e}^{inx}\right )=\sum_{n=0}^\infty c_n\,{\rm e}^{inx}\ .$$
 It has been proved in \cite{GG1} that \eqref{szego} is globally well posed on Sobolev spaces $H^s_+(\T ):=H^s(\T )\cap L^2_+(\T )$ for all $s\ge \frac 12$, with conservation of the $H^{\frac 12}$ norm. Recall that, for elements in $L^2_+(\T )$, the $H^s$ Sobolev norm reads
$$\Vert v\Vert _{H^s}^2=\sum_{n=0}^\infty (1+n)^{2s}|\hat v(n)|^2\ .$$
Furthermore, it turns out that \eqref{szego} enjoys an unexpected  Lax pair structure, discovered in \cite{GG1} and studied in \cite{GG2}, \cite{GG4}, \cite{GG6}.
More precisely, consider, for every $u\in H^{\frac 12}_+(\T )$, the Hankel operator $H_u:L^2_+(\T )\to L^2_+(\T )$ defined as
$$H_u(h)=\Pi (u\overline h)\ .$$
Notice that $H_u$ is an antilinear realization of the Hankel matrix $\Gamma_{\hat u}$, where, for every sequence  $\alpha =(\alpha _n)_{n\ge 0}$ of complex numbers, $\Gamma _\alpha $ denotes the operator on $\ell ^2(\Z_+)$ given by the infinite matrix $(\alpha _{n+p})_{n,p\ge 0}$. Indeed, if $\mathscr F$ denotes the Fourier transform $v\mapsto \hat v$ between $L^2_+(\T )$  and $\ell ^2(\Z _+)$,  it easy to check that
$$\mathscr FH_u\mathscr F^{-1}=\Gamma _{\hat u}\circ \mathcal C\ ,$$
where $\mathcal C$ denotes the complex conjugation. The Lax pair identity then reads as follows, see \cite{GG1}. If $s>\frac 12$ and $u$ is a $H^s_+$ solution of \eqref{szego}, then 
$$\frac{dH_u}{dt}=[B_u,H_u]\ ,$$
where $B_u$ is a linear antiselfadjoint operator depending on $u$. As a consequence, there exists a one parameter family $U(t)$ of unitary operators on $L^2_+(\T )$ such that
$$\forall t\in \R \ ,\ H_{u(t)}=U(t)H_{u(0)}U(t)^*\ .$$
In particular, $H_{u(t)}^2=U(t)H_{u(0)}^2U(t)^*\ .$ Notice that $H_u^2$ is a linear positive operator on $L^2_+(\T )$, and that
$$\mathscr FH_u^2\mathscr F^{-1}=\Gamma _{\hat u}\Gamma _{\hat u}^*\ ,$$
thus $H_u^2$ is a trace class operator as soon as $u\in H^{\frac 12}_+(\T )$, with
$${\rm Tr}(H_u^2)=\sum_{n=0}^\infty (1+n)|\hat v(n)|^2=\Vert v\Vert _{H^{\frac 12}}^2\ .$$
Consequently,  apart from $0$, the spectrum of $H_u^2$ is made of eigenvalues, which are conservation laws of \eqref{szego}.\\
In fact, a second Lax pair for \eqref{szego} holds \cite{GG2}, which concerns the operator $K_u:=S^*H_u=H_uS=H_{S^*u}$, where $S$ denotes the shift operator on $L^2_+(\T )$, namely  multiplication by ${\rm e}^{ix}$. Operator $K_u$ is also a Hankel operator, 
$$\mathscr FK_u\mathscr F^{-1}=\tilde \Gamma _{\hat u}\circ \mathcal C\ ,$$
where $\tilde \Gamma _\alpha $ denotes the shifted Hankel matrix $(\alpha _{n+p+1})_{n,p\ge 0}$.
Again, it is possible to prove that 
$$\forall t\in \R \ ,\ K_{u(t)}=V(t)K_{u(0)}V(t)^*\ ,$$
for some one parameter family $V(t)$ of unitary operators on $L^2_+(\T )$, and consequently that the eigenvalues of $K_u^2$ are conservation laws of \eqref{szego}. Denote by $(\rho_j^2)_{j\ge 1}$ the positive eigenvalues of $H_u^2$ and by $(\sigma_k^2)_{k\ge 1}$  the positive eigenvalues of $K_u^2$, so that the $\rho_j$'s --- resp. the $\sigma_k$'s --- are the singular values of $\Gamma_{\hat u}$ --- resp. of $\tilde \Gamma_{\hat u}$.  In view of the identity
$$K_u^2=H_u^2-(.\vert u)_{L^2}u$$
and of the min--max theorem, there holds the following interlacement property,
$$\rho_1\ge \sigma_1\ge \rho_2\ge \sigma_2\ge \dots $$
\subsection{The spectral transform} If $u$ belongs to a dense $G_\delta $ subset $H^{\frac 12}_{+,{\rm gen}}(\T )$ of $H^{\frac 12}_+(\T )$, one can establish -- see \cite{GG2} --- that 
$$\rho_1> \sigma_1> \rho_2> \sigma_2> \dots $$
We set 
$$s_{2j-1}=\rho_j\ ,\ s_{2k}=\sigma_k \ ,\ j,k\ge 1\ .$$
The $s_r$'s are called the singular values of the pair $(H_u,K_u)$. Of course 
$$\sum_{r=1}^\infty s_r^2={\rm Tr} (H_u^2)+{\rm Tr}(K_u^2)=\sum_{n=0}^\infty (1+2n)|\hat u(n)|^2<\infty \ .$$
Conversely, given a square summable --- strictly --- decreasing sequence $(s_r)_{r\ge 1}$ of  positive numbers, the set of $u\in H^{\frac 12}_+$ such that the $s_r$'s are the singular values of the pair $(H_u,K_u)$, in the above sense, is an infinite dimensional torus $\mathcal T((s_r)_{r\ge 1})$ \cite{GG2} of $H^{\frac 12}_+(\T )$. This torus is parametrised by the following explicit representation \cite{GG6}, where we classically identify functions of $L^2_+(\T)$ with holomorphic functions $u=u(z)$ on the unit disc such that
$$\sup_{r<1}\int_0^{2\pi}|u(r{\rm e}^{ix})|^2\, dx <\infty \ .$$
The current element of the infinite dimensional torus $\mathcal T((s_r)_{r\ge 1})$ is then given by 
\begin{equation}\label{uCN}
 u(z)=\lim_{N\to \infty}\langle \mathscr C_N(z)^{-1}({\bf 1}_N), {\bf 1}_N\rangle\  ,\ |z|<1\ ,\ {\bf 1}_N:=\left(\begin{array}{c}1\\
\vdots\\
1\end{array}\right)\in \C^N\ ,
\end{equation}
where 
\begin{equation}\label{CN}
\mathscr C_N(z):=\left (\frac{s_{2j-1}{\rm e}^{i\psi_{2j-1}}-z s_{2k}{\rm e}^{i\psi_{2k}}}{s_{2j-1}^2-s_{2k}^2}  \right )_{1\le j,k\le N}\ .
\end{equation}
and $(\psi_r)_{r\ge 1}\in \T^\infty $ is an arbitrary sequence of angles. Furthermore, the evolution of the new variables $(s_r,\psi_r)_{r\ge 1}$ through the dynamics of \eqref{szego} is given by 
$$\frac{ds_r}{dt}=0\ ,\ \frac{d\psi_r}{dt}=s_r^2\ ,\ r=1,2,\dots $$
{A} natural question is then the description of the regularity of $u$ in these new variables. A first type of answer to this question is provided by results due to Peller and Semmes --- see e.g. \cite{Pe}---, which characterise the Schatten classes 
$$\sum_{r\ge 1}s_r^p<\infty \ ,\ 0<p<\infty \ ,$$
in terms of the Besov spaces 
$$\sum_{j=0}^\infty 2^j \int_\T |\Delta _j u|^p \, dx<\infty \ ,$$
where $(\Delta_ju)_{j\ge 0}$ denotes the dyadic blocks of $u$. In particular, if $u$ is smooth, then $(s_r)_{r\ge 1}$ satisfies
$$\sum_{r=1}^\infty s_r^p<\infty \ ,\ \forall p<\infty \ .$$
However, the latter condition is far from being sufficient to control high regularity of $u$. In fact, Sobolev regularity $H^s$ for $s>\frac 12$ cannot be easily described by the variables $(s_r,\psi_r)_{r\ge 1}$, as shown by the following result.
\begin{theorem}[\cite{GG6}]\label{Ast}
There exists a dense $G_\delta $ subset of initial data in $$C^\infty_+(\T):=\bigcap_{s}H^s_+(\T )$$ such that the corresponding solutions of \eqref{szego} satisfies,
for every $s>\frac 12$,
\begin{eqnarray*}
\forall M\ge 1\ ,\ &&\limsup_{t\to \infty }\frac{\Vert u(t)\Vert_{H^s}}{|t|^M}=+\infty \\
&&\liminf_{t\to \infty }\Vert u(t)\Vert_{H^s}<\infty \ .
\end{eqnarray*}
\end{theorem}
In other words, in the $(s_r,\psi_r)_{r\ge 1}$ representation, the size of the high Sobolev norms may strongly depend on the angles $(\psi_r)_{r\ge 1}$. 
The goal of this paper is to investigate this phenomenon in more detail.
\subsection{Overview of the results}
Our first result claims that  generic smooth functions $u$  are located on a  torus  $\mathcal T((s_r)_{r\ge 1})$ containing also very singular functions. 
\begin{theorem}\label{meltingpot}
There exists a dense $G_\delta $ subset $\mathscr G$ of $C^\infty_+(\T )$ such that every element $u$ of  $\mathscr G$ belongs to $H^{\frac 12}_{+, {\rm gen}}(\T)$, and the infinite dimensional torus $\mathcal T((s_r)_{r\ge 1})$ passing through $u$ has a dense $G_\delta $ subset --- for  the $H^{\frac 12}$ topology--- which is disjoint of $H^s$ for every $s>\frac 12$. 
 \end{theorem}
 Theorem \ref{meltingpot} states that, on the tori $\mathcal T((s_r)_{r\ge 1})$ passing through generic smooth functions, the regularity changes dramatically from $C^\infty $ to the outside of $H^s$ for every $s>\frac 12$. Of course, this result can be seen as a  natural extension of Theorem \ref{Ast} recalled above, of which we use the weaker form that 
 tori $\mathcal T((s_r)_{r\ge 1})$ passing through generic smooth functions are unbounded in $H^s$ for every $s>\frac 12$. However, in order to find singular functions on these tori , we combine it with a structure property of these tori, which we think has its own interest.
 \begin{lemma}\label{alternative}
 Let  $s>\frac 12$ and let $(s_r)_{r\ge 1}$ be a square summable decreasing sequence of positive numbers such that the numbers $\{ s_r^2, r\ge 1 \}$ are linearly independent on $\Q $. Then we have the following alternative.
 \begin{itemize}
 \item Either $\mathcal T((s_r)_{r\ge 1})$ is  a bounded subset of $H^s$,
 \item Or $\mathcal T((s_r)_{r\ge 1})\setminus H^s$ is a dense $G_\delta $ subset of $\mathcal T((s_r)_{r\ge 1})$ for the $H^{\frac 12}$ topology.
 \end{itemize}
 \end{lemma}
The point of Theorem \ref{meltingpot} is that, even for fast decaying singular values $(s_r)$, the regularity of $u$ may be spoiled by the relative smallness of 
the gaps $s_r-s_{r+1}$ with respect to $s_r$. In fact,  if 
$$u_N(z)=\langle \mathscr C_N(z)^{-1}({\bf 1}_N),{\bf 1}_N\rangle \ ,\ \psi_r=0, r=1,2,\dots \ ,$$
with the notation introduced above, then, using the positivity property of the Hankel matrices $\Gamma_{\hat u_N}$ and $\tilde \Gamma_{\hat u_N}$
equivalent to $\psi_r=0$ for all $r$ --- see  \cite{GG3}, \cite{GP1}---, we prove in the appendix that 
\begin{equation}\label{estC1}
\Vert u_N\Vert_{C^1(\T)}\geq \sum_{j=1}^N\frac{s_ {2j-1}s_{2j}}{s_{2j-1}-s_{2j}}\ .
\end{equation}
It is then easy to find  fast decaying sequences $(s_r)$ such that the above right hand side tends to infinity as $N$ goes to infinity, which implies that $(u_N)$ is unbounded in $C^1(\T )$. However, at this stage we do not know how to conclude that $u$ is not in $C^1(\T)$.

Our two other results state some uniform analytic regularity for  tori $\mathcal T((s_r)_{r\ge 1})$
where the sequence $(s_r)_{r\ge 1}$ satisfies some specific property ruling out the phenomenon of small gaps.
\begin{theorem}\label{sous_geometrique}
For every $\rho >0$, there exists $\delta _0>0$ such that, for any $\delta \in (0,\delta _0)$, if
$$\forall r\ge 1\ ,\ s_{r+1}\leq \delta s_r\ ,$$
all functions $u\in \mathcal T((s_r)_{r\ge 1})$ are holomorphic and uniformly bounded  in the disc $|z| <1+\rho$. Consequently, for any initial datum corresponding to some of these functions, the solution of the cubic Szeg\H{o} equation \eqref{szego} is analytic in the disc of radius $1+\rho$ for all time, and is uniformly bounded in this disc. In particular, the trajectory is bounded in $C^\infty(\T )$.
\end{theorem}
Theorem \ref{sous_geometrique} applies in particular to geometric sequences $s_r={\rm e}^{-rh}$ for $h>0$ large enough. Our last result explores in more detail the case of geometric sequences $s_r={\rm e}^{-rh}$, where $h>0$ is arbitrary. In this case, we still obtain some uniform analytic regularity, but with a constraint on the angles $\psi_r$.
\begin{theorem}\label{totally_geometric}
Let $h>0$ and $\theta \in \R$. Assume $(s_r)$ is given by $s_r={\rm e}^{-rh}$ and $(\psi_r)$ by $\psi_r=r\theta h$. Then there exists $\rho>0$ such that the corresponding elements of  $\mathcal T((s_r)_{r\ge 1})$ are holomorphic and uniformly bounded in the disc $|z| <1+\rho$. 
\end{theorem}
We do not know whether or not geometric tori are embedded into the space of analytic functions on $\T $. What we are able to prove  is that, for transcendental $\gamma$ , we have the following alternative, 
 \begin{itemize}
 \item Either there exists $\rho>0$ such that every element of $\mathcal T((\gamma^r)_{r\ge 1})$ is holomorphic on the disc $|z|<1+\rho $, with a uniform bound,
 \item Or  the non-analytic elements of $\mathcal T((\gamma^r)_{r\ge 1})$ form a dense $G_\delta $ subset of $\mathcal T((\gamma^r)_{r\ge 1})$ for the $H^{\frac 12}$ topology.
 \end{itemize} {This is a special case of an extension of Lemma \ref{alternative} to analytic  regularity (see Lemma \ref{alternativeanalytic}).}

\subsection{Open problems}
In view of the above theorems, the most natural open question is certainly to decide whether Theorem  \ref{sous_geometrique} can be generalised to any parameter $\delta <1$. In particular, {as we questioned above, if $0<\gamma <1$, is it true that the infinite dimensional torus $\mathcal T((\gamma ^r)_{r\ge 1})$} is included into the space of analytic functions on $\T $ ? 

Another question connected to Theorem \ref{meltingpot} relies on estimate \eqref{estC1}. Assuming that 
$$\sum_{j=1}^\infty \frac{s_{2j-1}s_{2j}}{s_{2j-1}-s_{2j}}=\infty \ ,$$
can one infer that the function $u\in \mathcal T((s_r)_{r\ge 1})$ characterised by $\psi_r=0, r=1,2,\dots \ ,$ is not $C^1$ on $\T $ ? In view of Lemma \ref{alternative}, this would imply, if moreover the $s_r^2$ are linearly independent on $\Q $, that most of the points on this torus would be singular --- say, not in $H^2$. Then it would be interesting to draw the consequences of this property for long term behaviour of solutions of the cubic Szeg\H{o} equation on this torus.
\subsection{Organisation of the paper}
The proof of Theorem \ref{meltingpot} is provided in Section \ref{colourful} after reducing to Lemma \ref{alternative} and Theorem \ref{Ast}. The proof of Lemma \ref{alternative} combines a Baire category argument and some elementary ergodic argument for the cubic Szeg\H{o} flow.
Section \ref{delta} is devoted to the proof of Theorem \ref{sous_geometrique}, which is based on brute force estimates on matrices $\mathscr C_N(z)$. In Section \ref{gamma}, we prove Theorem \ref{totally_geometric} by a different approach relying on the theory of Toeplitz operators and a theorem by Baxter which reduces our analysis to proving that the restriction to  $\T $ of a meromorphic function given by an explicit  series, has no zero and has index $0$, which can be realised using  some elementary complex analysis and the Poisson summation formula.  Finally, the estimate \eqref{estC1} is derived in Appendix \ref{C1} from an explicit calculation using Cauchy matrices, in the spirit of \cite{GG6} and of \cite{GP2}.
\section{The melting pot property}\label{colourful}
In this section, we prove Theorem \ref{meltingpot}. First we reduce the proof to Lemma \ref{alternative} by the following classical argument.
\begin{lemma}
The set of $u\in C^\infty_+(\T )\cap H^{\frac 12}_{+{\rm gen}}(\T )$  such that the squares $s_r(u)^2, r\ge 1$ of the  singular values $s_r(u)$ are linearly independent on $\Q$, is a dense $G_\delta $ subset of $C^\infty_+(\T )$. 
\end{lemma}
\begin {proof}
From the proof of \cite{GG2}, Lemma 7, we already know that $C^\infty_+(\T )\cap H^{\frac 12}_{+{\rm gen}}(\T )$ is a dense $G_\delta $ subset of $C^\infty_+(\T)$. In fact, we can slightly modify the proof as follows. For every $N$, consider the open subset
$\mathcal O_N$ made of functions  $u\in C^\infty_+(\T )$ such that the first singular values of $H_u$ and $K_u$ satisfy 
$$\rho_1(u)>\sigma_1(u)>\rho_2(u)>\sigma_2(u)>\dots \rho_N(u)>\sigma_N(u)\ ,$$
and such that any non trivial linear combination of $$\rho_1(u)^2,\sigma_1(u)^2,\rho_2(u)^2,\sigma_2(u)^2,\dots \rho_N(u)^2,\sigma_N(u)^2$$ with integer coefficients in $[-N,N]$, is not zero. Approximating elements of $C^\infty_+(\T )$ by rational functions, and using the inverse spectral theorem of \cite{GG2} for rational functions, we easily obtain that $\mathcal O_N$ is dense. The conclusion follows from Baire's theorem. 
\end{proof}
Intersecting the dense $G_\delta $ subset of $C^\infty_+(\T )$ provided by this lemma with the one provided by Theorem \ref{Ast} --- or its weaker form, saying that the corresponding Szeg\H{o} trajectories are unbounded in every $H^s$, $s>\frac 12$, --- we observe that Theorem \ref{meltingpot} is a consequence of Lemma 
\ref{alternative}, which we restate for the convenience of the reader.
\begin{lemma}
 Let  $(s_r)_{r\ge 1}$ be a square summable decreasing sequence of positive numbers such that the numbers $\{ s_r^2, r\ge 1 \}$ are linearly independent on $\Q $ and let $s>\frac 12$. Then we have the following alternative.
 \begin{itemize}
 \item Either $\mathcal T((s_r)_{r\ge 1})$ is  a bounded subset of $H^s$,
 \item Or $\mathcal T((s_r)_{r\ge 1})\setminus H^s$ is a dense $G_\delta $ subset of $\mathcal T((s_r)_{r\ge 1})$ for the $H^{\frac 12}$ topology.
 \end{itemize}
 \end{lemma}
\begin{proof}
Recall \cite{GG2,GG6} that, for the $H^{\frac 12}$ topology,  $\mathcal T((s_r)_{r\ge 1})$ is homeomorphic to the infinite dimensional tori $\T^\infty $, endowed with the product topology,  through the parametrisation given by 
\eqref{uCN} and \eqref{CN}. In particular, it is a compact metrizable space. For every $s>\frac 12$, the function
$$\Vert v\Vert_{H^s}=\left (\sum_{n=0}^\infty (1+n)^{2s}|\hat v(n)|^2   \right )^{\frac 12}$$
is lower semi--continuous on $\mathcal T((s_r)_{r\ge 1})$. For every positive integer $\ell $, consider
$$F_\ell =\{ v\in \mathcal T((s_r)_{r\ge 1}): \Vert v\Vert_{H^s}\leq \ell \}\ .$$
Then $F_\ell $ is a closed subset of $\mathcal T((s_r)_{r\ge 1})$, and the complement of the union of the $F_\ell $ is precisely $\mathcal T((s_r)_{r\ge 1})\setminus H^s$. Hence, by the Baire theorem, either this set is a dense $G_\delta $ subset of $\mathcal T((s_r)_{r\ge 1})$, or there exists $\ell \ge 1$ such that
$F_\ell $ has a nonempty interior. Assume that some $F_\ell $ has a nonempty interior, and let us show that $\mathcal T((s_r)_{r\ge 1})$ is a bounded subset of $H^s$. Let $(\psi_r^0)_{r\ge 1}\in \T^\infty $ such that the corresponding point $v^0$ in $\mathcal T((s_r)_{r\ge 1})$ lies in the interior of $F_\ell $. In view of the product topology on $\T ^\infty $, there exists some integer $N\ge 1$ and some $\e >0$ such that all the elements of $\mathcal T((s_r)_{r\ge 1})$ corresponding to 
$$\psi_r\in ]\psi_r^0-\e ,\psi_r^0+\e [\ ,\ r=1,\dots ,N\ ,$$
form an open set $U$ contained into $F_\ell $. At this stage we appeal to the number theoretic assumption on the $s_r^2$, which we use classically under the form that
the trajectory
$$\{ (\psi_r^0+ts_r^2)_{r=1,\dots ,N}\ ,\ t\in \R\}$$
is dense into the torus $\T ^N$. Since, as recalled in the introduction,  this trajectory is precisely the projection of the trajectory of the cubic Szeg\H{o} flow $\Phi_t$ on the first $N$ components, we infer that  every element of $\mathcal T((s_r)_{r\ge 1})$ is contained into some open set $\Phi_t(U)$. Since the cubic Szeg\H{o} equation is wellposed on $H^s$ \cite{GG1}, we infer that $\mathcal T((s_r)_{r\ge 1})$ is covered by the union of the interiors of the $F_m$ for $m\geq 1$. By compactness, it is covered by a finite union, which precisely means that $\mathcal T((s_r)_{r\ge 1})$ is bounded in $H^s$.
\end{proof}
 As stated in the introduction for geometric sequences, {the following} analogous result holds in the analytic setting.
\begin{lemma}\label{alternativeanalytic}
Let  $(s_r)_{r\ge 1}$ be a square summable decreasing sequence of positive numbers such that the numbers $\{ s_r^2, r\ge 1 \}$ are linearly independent on $\Q $. Then we have the following alternative.
 \begin{itemize}
 \item Either there exists $\rho>0$ such that every element of $\mathcal T((s_r)_{r\ge 1})$ is holomorphic on the disc $|z|<1+\rho $, with a uniform bound,
 \item Or  the non-analytic elements of $\mathcal T((s_r)_{r\ge 1})$ form a dense $G_\delta $ subset of $\mathcal T((s_r)_{r\ge 1})$ for the $H^{\frac 12}$ topology.
 \end{itemize}
\end{lemma}
\begin{proof}
 The proof is an adaptation of the preceding one (Lemma \ref{alternative}) to the analytic setting. As, from \cite{GGT},  the cubic Szeg\H{o} equation propagates analyticity, the result follows from the Baire theorem applied to the closed sets 
$$F_\ell :=\{ v\in \mathcal T((s_r)_{r\ge 1}): \sum_{n=0}^\infty {\rm e}^{\frac n\ell}|\hat v(n)|\leq \ell \}$$
for $\ell \geq 1$.
\end{proof}
\section{Example of bounded analytic tori}\label{delta}
In this section, we prove Theorem \ref{sous_geometrique}.
\begin{proof}
Let $u\in \mathcal T((s_r)_{r\ge 1})$. Recall that 
$$u=\lim_{N\to \infty}u_N\text{ where }u_N(z):=\langle \mathscr C_N(z)^{-1} {\bf 1}_N\vert {\bf 1}_N\rangle,$$
$$\mathscr C_N(z):=\left(\frac{s_{2j-1}{\rm e}^{i\psi_{2j-1}}-zs_{2k}{\rm e}^{i\psi_{2k}}}{s_{2j-1}^2-s_{2k}^2}\right)_{1\le j,k\le N}$$ and
$$ {\bf 1}_N=\left(\begin{array}{c}1\\
\vdots\\
1\end{array}\right)\in \C^N.$$
Our assumption is  
$$
s_{r+1}=\e_r s_r\ ,\  r\ge 1\ ,$$
where the sequence $(\e_r)_{r\ge 1}$ satisfies $$0<\e_r\le \delta\text{ for some }\delta<1.$$
Our aim is  to prove that, for $\delta$ sufficiently small, the functions $u_N$ are holomorphic and uniformly bounded in some disc of radius $1+\rho$, where $\rho>0$, independently of $N$. Our strategy is to use that $\mathscr C_N(0)$ is related to a Cauchy matrix, and hence, that an explicit formula for its inverse is known. 
We write
 \begin{eqnarray*}
 \mathscr C_N(z)=\mathscr C_N(0)-z\dot{\mathscr C_N}=\mathscr C_N(0)(I-z\mathscr C_N(0)^{-1}\dot {\mathscr C_N})
 \end{eqnarray*}
 where 
 $$\dot{\mathscr C_N}:= \left(\frac{s_{2k}{\rm e}^{i\psi_{2k}}}{s_{2j-1}^2-s_{2k}^2}\right)_{1\le j,k\le N}$$
and we establish the following lemma.
\begin{lemma}\label{Estimates}
For  any $0<\delta<1$, there exists some  constant $C_\delta >0$ such that, for any $N\ge 1$,
\begin{equation}
\label{C_N(0)^{-1}}
\sum_{j,k}\left|(\mathscr C_N(0)^{-1})_{j,k}\right|\le C_\delta \, s_1\ .
\end{equation}
There exists a universal constant $A>0$ such that, for $\delta\in (0,\frac 12)$  and for  any $N\ge 1$,   
\begin{equation}
\label{R_N}
\Vert \mathscr C_N(0)^{-1}\dot{\mathscr C}_N\Vert_{\ell^1\to \ell^1} \le { A\, \delta }\ .
  \end{equation}
\end{lemma}
Let us assume this lemma proved. Take $\rho >0$, and choose $\delta _0$ such that $A\delta _0(1+\rho )\leq \frac 12$. Hence, for any $\delta \in (0,\delta _0)$, from estimate (\ref{R_N}),  $$(I-z\mathscr C_N(0)^{-1} \dot{\mathscr C_N})$$ is invertible for any $z$ with $|z|<1+\rho $ and its inverse $R_N(z)$ is analytic and has uniformly bounded norm 
for any $z$ with $|z|< 1+\rho $. Indeed, for any $N\ge 1$, any $z$ with $|z|<1+\rho $, by the Neumann series identity,
\begin{equation}\label{estim R_N}\Vert R_N(z)\Vert_{\ell^1\to \ell^1}\leq \sum_{k=0}^\infty |z|^k\Vert  \mathscr C_N(0)^{-1}\dot{\mathscr C_N}\Vert_{\ell^1\to \ell^1}^k  \leq  \sum_{k=0}^\infty 2^{-k}\leq 2\ .\end{equation}
Writing $$\mathscr C_N(z)^{-1}=(I-z(\mathscr C_N(0)^{-1} \dot{\mathscr C_N}))^{-1}\mathscr C_N(0)^{-1} =R_N(z)(\mathscr C_N(0))^{-1}$$
we get $$u_N(z)=\langle R_N(z)\mathscr C_N(0)^{-1}({\bf 1}_N), {\bf 1}_N\rangle .$$
Using (\ref{C_N(0)^{-1}}) and (\ref{estim R_N}),
we conclude that the series defining $u_N$ converges uniformly for $|z|< 1+\rho $. Hence $u_N$ is  analytic and uniformly bounded in the disc of radius $1+\rho$.
We infer that $u$ is as well analytic in the disc of radius $1+\rho$ and bounded on this disc. This completes the proof of Theorem \ref{sous_geometrique}, modulo Lemma \ref{Estimates}.
\end{proof}
\subsection{Proof of Lemma \ref{Estimates}}
Notice that 
$$\mathscr C_N(0)={\rm diag}\left ( s_{2j-1}{\rm e}^{i\psi_{2j-1}}\right )\mathscr T\ ,\ \mathscr T:=\left(\frac{1}{s_{2j-1}^2-s_{2k}^2}\right)_{1\le j,k\le N}\ .
$$
  Since $\mathscr T$ is a Cauchy matrix, its inverse is explicitly known, so the inverse of $\mathscr C_N(0) $ is given by
  $$\mathscr C_N(0)^{-1}=\left(\frac{(-1)^{j+k+N}\alpha_j^{(N)}\beta_k^{(N)}}{ s_{2j-1}^2-s_{2k}^2}\frac 1{s_{2j-1}{\rm e}^{i\psi_{2j-1}}}\right)_{1\le k,j\le N}$$
  where 
\begin{eqnarray*}
\alpha_j^{(N)}&:=&\frac{\prod_\ell (s_{2j-1}^2-s_{2\ell}^2)}{\prod_{\ell<j}(s_{2\ell-1}^2-s_{2j-1}^2)\prod_{\ell>j}(s_{2j-1}^2-s_{2\ell-1}^2)}\\
\beta_k^{(N)}&:=&\frac{\prod_\ell (s_{2\ell-1}^2-s_{2k}^2)}{\prod_{\ell<k}(s_{2\ell}^2-s_{2k}^2)\prod_{\ell>k}(s_{2k}^2-s_{2\ell}^2)}.
\end{eqnarray*}
  In particular,
  \begin{eqnarray*}
|\alpha_j^{(N)}|&=&\prod_{\ell<j}\frac{s_{2\ell}^2}{s_{2\ell-1}^2}\prod_{\ell<j}\left (\frac{ 1-\prod_{r=2\ell}^{2j-2}\e_r ^2}{1-\prod_{r=2\ell-1}^{2j-2}\e_r^2}\right )\prod_{\ell >j}\frac{(1-\prod_{r=2j-1}^{2\ell-1}\e_r^2 )}{(1-\prod_{r=2j-1}^{2\ell-2}\e_r^2)}s_{2j-1}^2(1-\e_{2j-1}^2)\\
&\leq &\prod_{\ell<j}\e_{2\ell -1}^2 \, \frac{s_{2j-1}^2}{\prod _{m=1}^\infty (1-\delta ^{4m})} .
\end{eqnarray*}
{Indeed, in the first line above, the factors in the second product are bounded by $1$, while, in the third product, the $\ell$-- factor is bounded by $\frac{1}{1-\delta^{4(\ell -j)}}\ .$} Similarly, we have
\begin{eqnarray*}
|\beta_k^{(N)}|
&=&\prod_{\ell<k}\frac{s_{2\ell-1}^2}{s_{2\ell}^2}\prod_{\ell >k}\left (\frac{1-\prod_{r=2k}^{2\ell-2}\e_r^2 }{1-\prod_{r=2k}^{2\ell-1}\e_r^2}\right )\prod_{\ell<k}\frac{ (1-\prod_{r=2\ell-1}^{2k-1}\e_r^2 )}{(1-\prod_{r=2\ell}^{2k-1}\e_r^2)}s_{2k-1}^2(1-\e_{2k-1}^2)\\
&\leq&\prod_{\ell<k}\frac 1{ \e_{2\ell-1}^2}\, \frac{s_{2k-1}^2}{\prod_{m=1}^\infty (1-\delta ^{4m})}
\end{eqnarray*}
Setting
$$B_\delta =\frac{1}{\prod_{m=1}^\infty (1-\delta ^{4m})^2}\ ,$$
we obtain
 \begin{eqnarray*}
 \left|\left(\mathscr C_N(0)^{-1}\right)_{kj}\right|&\leq& B_\delta \frac{ s_{2j-1}s_{2k-1}^2}{|s_{2j-1}^2-s_{2k}^2|}\prod_{\ell<j}\e_{2\ell -1}^2\prod_{\ell<k}\frac 1{ \e_{2\ell-1}^2}\\
 &\leq& B_\delta \left\{\begin{array}{lll}
 &\frac 1{1-\delta^{4(k-j)+2}}\frac{s_{2k-1}^2}{s_{2j-1}}\prod_{j\le \ell<k}\frac 1{ \e_{2\ell-1}^2}&\text{ if }j<k\\
 &\frac 1{1-\delta^{2}}s_{2j-1}&\text{ if }j=k\\
 &\frac 1{1-\delta^{4(j-k-1)+2}}s_{2j-1}\frac{s_{2k-1}^2}{s_{2k}^2}\prod_{k\le \ell<j}\e_{2\ell-1}^2&\text{ if }j>k.\end{array}\right.
 \end{eqnarray*}
 To summarize,
 \begin{equation}\label{Estimate coeff}
 \left|\left(\mathscr C_N(0)^{-1}\right)_{kj}\right| \leq  \frac{B_\delta}{1-\delta^2} \, s_{2j-1}\left\{\begin{array}{lll}
 &{\delta^{2(k-j)}}&\text{ if }j<k\\
 &1&\text{ if }j=k, k+1\\
 &{\delta^{2(j-k-1)}}&\text{ if }j>k+1\end{array}\right.\\
\end{equation} 
In particular, it gives 
$$\sum_{k,j} \left|(\mathscr C_N(0)^{-1})_{jk} \right| \leq \frac{2 B_{\delta}}{(1-\delta ^2)^2}\sum_{j\le N} s_{2j-1}\leq \frac{{2}B_\delta \, s_1}{(1-\delta ^2)^3}\ .$$

 This proves estimate (\ref{C_N(0)^{-1}}).\\
For the second estimate, one has to consider
$$\Vert \mathscr C_N(0)^{-1}\dot{\mathscr C_N}\Vert_{\ell^1\to \ell^1}\leq \sup_{\ell}\sum_{k}|(\mathscr C_N(0)^{-1} \dot{\mathscr C_N})_{k\ell}|.$$
Recall that
 $$\dot{\mathscr C_N}:= \left(\frac{s_{2\ell }{\rm e}^{i\psi_{2\ell}}}{s_{2j-1}^2-s_{2\ell }^2}\right)_{1\le j,\ell \le N}.$$
 In particular,
 $$\left |(\dot{\mathscr C_N})_{j\ell}\right | \leq \frac{1}{1-\delta ^2}\left\{\begin{array}{lll} \frac{s_{2\ell}}{s_{2j-1}^2}&\text{ if } j\le \ell\\
 \frac 1{s_{2\ell}} &\text{ if } j\ge \ell +1\end{array}\right.$$
 As $(\mathscr C_N(0)^{-1} \dot{\mathscr C_N})_{k\ell}= \sum_j (\mathscr C_N(0)^{-1})_{kj} (\dot{ \mathscr C_N})_{j\ell}$, we get from
the preceding estimate (\ref{Estimate coeff}) on $|(\mathscr C_N(0)^{-1})_{kj}|$, 
\begin{itemize}
\item if $k>\ell$
\begin{eqnarray*}
  &&\left |(\mathscr C_N(0)^{-1} \dot{\mathscr C_N})_{k\ell}\right |\leq  \frac{B_\delta }{(1-\delta ^2)^2}\times \\ 
  &&\times \left (\sum_{j\le \ell}\frac{s_{2\ell}}{s_{2j-1}}\prod_{j\le r\le k-1} \e_{2r}^2 +\sum_{\ell+1\le j\le k}\frac{ s_{2j-1}}{s_{2\ell }}\prod_{j\le r\le k-1} \e_{2r}^2+\sum_{j\ge k+1}  \frac{s_{2j-1}}{s_{2\ell }}\prod_{k+1 \le r\le j-1}\e_{2r-1}^2\right )\\
  &\leq& \frac{B_\delta }{(1-\delta ^2)^2}\left (\sum_{j\le \ell}\delta^{2(\ell-j)+1}\delta^{2(k-j)}+\sum_{\ell+1\le j\le k}\delta^{2(j-\ell-1)+1}\delta^{2(k-j)}+ \sum_{j\ge k+1}\delta^{2(j-\ell -1)+1}\delta^{2(j-k-1)}\right )\\
 & \le& \delta  \frac{B_\delta }{(1-\delta ^2)^2}\left(2 \delta^{2(k-\ell)}\sum_{s\ge 0} \delta^{4s}+\sum_{\ell+1\le j\le k}\delta^{2(j-\ell-1)}\delta^{2(k-j)}\right)
    \end{eqnarray*}
 As $$\sum_{k,\atop k\ge \ell +1}  \sum_{\ell+1\le j\le k}\delta^{2(j-\ell-1)}\delta^{2(k-j)}=\sum_{j\ge \ell+1}\sum_{k\ge j}\delta^{2(j-\ell-1)}\delta^{2(k-j)}=\frac 1{(1-\delta^2)^2}$$ one gets
     $$\sum_{k ,\atop k>\ell} \left |(\mathscr C_N(0)^{-1} \dot{\mathscr C_N})_{k\ell}\right | \le\frac{ \delta B_\delta }{(1-\delta ^2)^4 }\left( \frac {1+3\delta^2}{1+\delta^2}\right).$$
\item If $k<\ell$
\begin{eqnarray*}
&&\left |(\mathscr C_N(0)^{-1} \dot{\mathscr C_N})_{k\ell}\right |  \leq  \frac{B_\delta }{(1-\delta ^2)^2}\times \\
&&\times \left (\sum_{j\le k}\frac{s_{2\ell}}{s_{2j-1}}\prod_{j\le r\le k-1} \e_{2r}^2 +\sum_{k+1\le j\le \ell}\frac{ s_{2\ell }}{s_{2j-1}}\prod_{k+1\le  r\le j-1}\e_{2r-1}^2+\sum_{j\ge \ell+1}  \frac{s_{2j-1}}{s_{2\ell }}\prod_{k+1\le r\le j-1}\e_{2r-1}^2\right )\\
 &\leq&\frac{B_\delta }{(1-\delta ^2)^2}\left (\sum_{j\le k}\delta^{2(\ell -j)+1}\delta^{2(k-j)}+\sum_{k+1\le j\le \ell}\delta^{2(\ell-j)+1}\delta^{2(j-k-1)}+\sum_{j\ge \ell+1}\delta^{2(j-\ell-1)+1}\delta^{2(j-k-1)}\right )\\
    &\leq&\frac{\delta B_\delta }{(1-\delta ^2)^2}\left (2\delta^{2(\ell-k)}\sum_{s\ge 0}\delta^{4s}+\sum_{k+1\le j\le \ell}\delta^{2(\ell-j)}\delta^{2(j-k-1)}\right )
    \end{eqnarray*}
     and, as before, 
   $$\sum_{k , \atop k<\ell} \left |(\mathscr C_N(0)^{-1} \dot{\mathscr C_N})_{k\ell}\right | \le \frac{ \delta\, B_\delta }{(1-\delta ^2)^4}\left( \frac {1+3\delta^2}{1+\delta^2}\right).$$ 
  \item For $k=\ell$, 
  \begin{eqnarray*}
  \left |(\mathscr C_N(0)^{-1} \dot{\mathscr C_N})_{kk}\right |  &\leq&\frac{B_\delta}{(1-\delta ^2)^2}\left (\sum_{j\le \ell}\frac{s_{2\ell}}{s_{2j-1}}\prod_{j\le r\le k-1} \e_{2r}^2 +\sum_{j\ge k+1}  \frac{s_{2j-1}}{s_{2\ell }}\prod_{k+1\le r\le j-1}\e_{2r-1}^2\right )\\
  &\leq&\frac{B_\delta}{(1-\delta ^2)^2}\left (\sum_{j\le \ell}\delta^{2(\ell-j)+1}\delta^{2(k-j)}+\sum_{j\ge k+1}\delta^{2(j-\ell-1)+1}\delta^{2(j-k-1)}\right )\\
  &\leq& \frac{B_\delta}{(1-\delta ^2)^2}\left (\sum_{j\le k}\delta^{4(k-j)+1}+\sum_{j\ge k+1}\delta^{4(j-k-1)+1}\right ) =2\frac{\delta B_\delta }{(1-\delta ^2)^2(1-\delta^4)}. 
  \end{eqnarray*}
  \end{itemize}
Eventually,  if $\delta \leq \frac 12$, say, we obtain, with a universal constant $A$,
$$\Vert (\mathscr C_N(0)^{-1} \dot{\mathscr C_N})\Vert_{\ell^1\to \ell^1}\leq \sup_\ell \sum_{k } \left |(\mathscr C_N(0)^{-1} \dot{\mathscr C_N})_{k\ell}\right | \leq A\, \delta  .$$ 
This completes the proof of Lemma \ref{Estimates}.

\section{The totally geometric spectral data}\label{gamma}
In this section, we consider the totally geometric case and prove Theorem \ref{totally_geometric}. 
Namely, for some fixed $h>0$ and $\theta \in \R$, we consider the symbol $u$ with spectral data $(s_r,\psi_r)$ with $s_r={\rm e}^{-rh}$ and $\psi_r=r\theta h$. In particular, $s_{r+1}=s_r{\rm e}^{-h}$ so that, for $h$ sufficiently large, it becomes a particular case of sub-geometric spectral data treated in Theorem \ref{sous_geometrique}. However, the result here does not require any smallness on $ {\rm e}^{-h}$.\\
Our strategy here is to use Toeplitz operators and a stability result from Baxter (1963). 
\subsection{Background on Toeplitz operators} Let us first introduce  some basic notation.
For  a continuous function $\Phi$ on $\T$, we denote by $T(\Phi)$ the Toeplitz operator of symbol $\Phi$ defined on $L^2_+(\T)$ by
$$T(\Phi)(f)=\Pi(\Phi f)$$ or equivalently the operator defined on $\ell^2(\N)$ by
$$(T(\Phi)((a_k)))_j:=\sum_{k=0}^\infty\hat \Phi(j-k)a_k,\; j\in \N.$$
For any integer $N$, we denote by $T_N(\Phi)$ the truncated operator defined by
$$T_N(\Phi):=\Pi_NT(\Phi)\Pi_N.$$
Here $$\Pi_N: \ell^2(\N)\to \ell^2(\N)$$ is the orthogonal projector :$$ (x_0,x_1,x_2,\dots)\mapsto (x_0,x_1,x_2,\dots,x_{N-1}).$$
The operator $T_N$ corresponds to the $N\times N$ truncated Toeplitz matrix $$(\hat\Phi(j-k))_{0\le j,k\le N-1}.$$
Recall that a sequence of $N\times N$ matrices $(A_N)_{N\ge 1}$ is said to be stable if there is an $N_0$ such that the matrices $A_N$ are
invertible for all $N\ge N_0$ and 
$$\sup_{N\ge N_0}\Vert A_N^{-1}\Vert_{\ell^2\to \ell^2} <\infty.$$
\begin{theorem} [\cite{Baxter}, \cite{BG}]\label{Baxter}
The sequence $(T_N(\Phi))_{N\ge 1}$ is
stable if and only if $T(\Phi)$ is invertible. 
\end{theorem}
Let us emphasize that the operators are considered as operators acting on $\ell^2(\N)$ or $L^2_+(\T)$ so that the stability is evaluated in the $\ell^2(\N)$ norm. The characterization of the invertibility of Toeplitz operators is well known. We recall it for the convenience of the reader.
\begin{theorem}\label{ContinuityToeplitz}
Let $\Phi$ be a continuous function on the unit circle. 
If $\Phi$ has index $0$ and does not vanish on the circle, then $T_\Phi$ is invertible on $L^2_+(\T)$.
Under these hypotheses, $\Phi={\rm e}^{\varphi}=\Phi_+\overline {\Phi_-}$ with
$$\Phi_+={\rm e}^{\Pi(\varphi)}\mbox{ and }\Phi_-={\rm e}^{\overline{(I-\Pi)(\varphi)}}$$
and the inverse of $T_\Phi$ is given by $T_{\Phi_+^{-1}}T_{\overline \Phi_-^{-1}}$.
\end{theorem}
As an immediate consequence, one gets the following characterization of the stability of truncated Toeplitz operators.
\begin{corollary}\label{Baxterbound}
Let $\Phi$ be a continuous function on the unit circle.
The sequence of truncated Toeplitz operators $(T_N(\Phi))$ is stable if and only if $\Phi$ has no zero on the unit circle and has index $0$.
\end{corollary}
We are going to use this argument to prove Theorem \ref{totally_geometric}.
\subsection{Totally geometric spectral data and Toeplitz operators}
We claim that in the case of totally geometric spectral data, the explicit formula giving $u_N$ involves the inverse of a truncated Toeplitz operator.
From direct computation, one has 
\begin{eqnarray*}\mathscr C_N(z)&=&\left( \frac{\omega ^{2j-1}-z\omega^{2k}}{|\omega|^{4j-2}-|\omega|^{4k}}\right)_{1\le j,k\le N}\\
&=&\left(\frac{1}{\overline \omega ^{2j-1}}\frac{1-z\omega^{2(k-j)+1}}{1-|\omega|^{4(k-j)+2}}\right)_{1\le j,k\le N}.
\end{eqnarray*}
where $\omega={\rm e}^{-h(1-i\theta)}$.
In that case, if $T_N(z)$ and $T_{N,r}(z)$ denote the matrices
$$T_N(z)=\left(\frac{1-z\omega^{2(k-j)+1}}{1-|\omega|^{4(k-j)+2}}\right)_{1\le j,k\le N}$$
and 
$$ T_{N,r}(z)=\left(r^{k-j}\frac{1-z\omega^{2(k-j)+1}}{1-|\omega|^{4(k-j)+2}}\right)_{1\le j,k\le N}$$
we get from our explicit formula, for any $r>0$,
$$
u_N(z)=<T_N(z)^{-1}(\overline \omega ^{2j-1}),{\bf 1}>=<T_{N,r}(z)^{-1}(r^{-j}\overline \omega ^{2j-1})_{1\le j\le N}),(r^k)_{1\le k\le N}>.
$$
We consider for $|\zeta|=r$, $|\omega|^2<r<1$, $z\in\C$, the symbol
$$\Phi(z,\zeta):=\sum_{\ell\in\Z}\frac{1-z\omega^{2\ell +1}}{1-|\omega|^{4\ell+2}}\zeta^\ell.$$
The transpose of the matrix $$\left(r^{k-j}\frac{1-z\omega^{2(k-j)+1}}{1-|\omega|^{4(k-j)+2}}\right)_{ j,k\ge 1}$$ corresponds to the matrix of the Toeplitz operator of symbol $$\Phi(z,r\cdot):\zeta\mapsto \Phi(z,r\zeta).$$
We  are going to prove the following result.
\begin{proposition}\label{invertibility}
There exist $|\omega|^2<r<1$ and $\rho>0$ such that the function $\zeta \mapsto \Phi(z,r\zeta)$ has no zero and has index $0$ on the the unit circle,  for every $z$ such that  $|z|<1+\rho$. 
\end{proposition}
Assuming this result  proved, we obtain by Corollary \ref{Baxterbound} that, uniformly in $z$, $|z|<1+\rho$, $\Vert T_{N,r}(z)^{-1}\Vert_{\ell^2\to \ell^2}$ is bounded (or more precisely the norm of its transpose is bounded). As $|\omega|^2<r<1$, we obtain that the sequence $(u_N(z))_N$ with
$$u_N(z)=<T_{N,r}(z)^{-1}(r^{-j}\overline \omega ^{2j-1})_{1\le j\le N}),(r^k)_{1\le k\le N}>$$
is uniformly bounded and converge to $u(z)$ for any $z$, $|z|<1+\rho$. We conclude as in the previous section. This ends the proof of Theorem \ref{totally_geometric}.
\vskip 0.25cm
It remains to prove Proposition \ref{invertibility}, which is the object of the next subsections.
As a preliminary, observe that, for $|\omega|^2<|\zeta|<1$, $\gamma=|\omega|^2$, $$\Phi(z,\zeta)=F_\gamma(\zeta)-z\omega F_\gamma(\zeta\omega^2),$$
where 
\begin{equation}\label{F_gamma}F_\gamma(\zeta)=\Phi(0,\zeta)=\sum_{j\in\Z}\frac{\zeta^j}{1-\gamma^{2j+1}},\; \gamma=|\omega|^2.
\end{equation}
We collect some basic properties of function $F_\gamma$ in the following lemma.
\begin{lemma}
The function $F_\gamma$ has a meromorphic extension in $\C\setminus\{0\}$ given by 
\begin{equation}\label{F_gamma2}
F_\gamma(\zeta)=\sum_{\ell\in\Z} \frac{\gamma^\ell}{1-\zeta \gamma^{2\ell}}.
\end{equation}
Its only poles in $\C\setminus\{0\}$ are the $\gamma^{2\ell}$'s, $\ell \in \Z$ and $F_\gamma(\gamma^{2\ell+1})=0$, $\ell \in \Z$. 
Furthermore
 \begin{equation}\label{properties}
F_\gamma\left(\frac 1\zeta\right)=-\zeta F_\gamma(\zeta),\; F_\gamma\left(\frac\zeta{\gamma^2}\right)=\gamma F_\gamma(\zeta).
\end{equation}
\end{lemma}
\begin{proof}
Let us give another expression of $F_\gamma$. By assumption, $$\gamma<|\zeta|< 1$$ hence, $|\zeta|>\gamma^2$ and
\begin{eqnarray*}
F_\gamma(\zeta)&=&\sum_{j=0}^\infty\frac{\zeta^j}{1-\gamma^{2j+1}}+\sum_{j=0}^\infty \frac{\zeta^{-j-1}}{1-\gamma^{-2j-1}}\\
&=&\sum_{j=0}^\infty \zeta^j\sum_{\ell=0}^\infty \gamma^{(2j+1)\ell}-\sum_{j=0}^\infty \frac{\zeta^{-j-1}\gamma^{2j+1}}{1-\gamma^{2j+1}}\\
&=&\sum_{\ell=0}^\infty \gamma^\ell\sum_{j=0}^\infty (\zeta\gamma^{2\ell})^j-\gamma\zeta^{-1}\sum_{\ell=0}^\infty \gamma^\ell\sum_{j=0}^\infty (\zeta^{-1}\gamma^{2\ell+2})^j\\
&=&\sum_{\ell=0}^\infty \frac{\gamma^\ell}{1-\zeta \gamma^{2\ell}}-\sum_{\ell=0}^\infty \frac{ \gamma^{\ell+1}}{\zeta-\gamma^{2\ell+2}}
\end{eqnarray*}
Hence, \begin{equation}\label{F_gamma2}
F_\gamma(\zeta)=\sum_{\ell\in\Z} \frac{\gamma^\ell}{1-\zeta \gamma^{2\ell}}.
\end{equation}
The other properties are elementary consequences of this equality.
\end{proof}
{
\begin{remark}\label{elliptic}
Set $\gamma ={\rm e}^{-\pi \tau}\ ,\ \tau >0$. From the second identity \eqref{properties}, we observe that the meromorphic function
$$G_\tau (w)={\rm e}^{2i\pi w}\left (F_\gamma ({\rm e}^{2i\pi w})\right )^2$$
satisfies 
$$\forall \lambda \in \Z +i\tau \Z\ ,\ ,G_\tau (w+\lambda )=G_\tau (w)\ ,$$
which means that $G_\tau $ is an elliptic function relative to the lattice $\Z +i\tau \Z$. Since $G_\tau $ has only double poles at the lattice points, with singularity
$$\frac{1}{(\zeta -1)^2}\sim -\frac{1}{4\pi^2w^2}$$
at $w=0$, and since it cancels at points $i\frac \tau 2+\Z +i\tau \Z$, we infer that
$$G_\tau (w)=-\frac{1}{4\pi^2}\left (\mathfrak{P}_\tau (w)-\mathfrak{P}_\tau \left (i\frac \tau 2\right ) \right )\ ,$$
where $$\mathfrak{P}_\tau (w)=\frac{1}{w^2}+\sum_{\lambda \ne 0}\left (\frac{1}{(w-\lambda )^2}-\frac{1}{\lambda ^2}\right )$$ denotes the Weierstrass $\mathfrak P$ function  relative to the lattice $\Z +i\tau \Z$. See e.g. \cite{SZ}
\end{remark}
}
\subsection{Ruling out the zeroes on the unit circle}
{In this section, we prove the following lemma.}
\begin{lemma}\label{ZeroesPhi}
There exists $\rho>0$ such that $\Phi(z,\zeta)$ does not vanish in a neighbourhood of the circle
 $|\zeta|=1$ for any $z$ such that $|z|\leq 1+\rho$.
\end{lemma}
Lemma \ref{ZeroesPhi} {is} a consequence of  the following result. 
\begin{lemma}\label{NonVanishing}
For every $\gamma \in (0,1)$,
$$\gamma ^{1/2}\max _{|\zeta|=\gamma}|F_\gamma (\zeta)|<\min_{|\zeta|=1}|F_\gamma (\zeta)|\ .$$
\end{lemma}
\begin{proof}
First of all we rewrite both sides of the above inequality. If $\zeta={\rm e}^{i\theta}$,
\begin{eqnarray*}
F_\gamma (\zeta)&=&\sum_{k=0}^\infty \frac{\gamma ^k}{1-\gamma ^{2k}{\rm e}^{i\theta}}+\sum_{\ell =1}^\infty \frac{\gamma ^{-\ell}}{1-\gamma ^{-2\ell }{\rm e}^{i\theta}}\\
&=&\sum_{k=0}^\infty \frac{\gamma ^k}{1-\gamma ^{2k}{\rm e}^{i\theta}}+\sum_{\ell=1}^\infty \frac{\gamma ^{\ell }}{\gamma ^{2\ell }-{\rm e}^{i\theta}}\\
&=&\frac{1}{1-{\rm e}^{i\theta}}+\sum_{\ell=1}^\infty\frac{\gamma^{\ell}(1+\gamma^{2\ell})(1-{\rm e}^{-i\theta})}{1+\gamma^{4\ell}-2\gamma^{2\ell}\cos\theta}\\
&=&(1-{\rm e}^{-i\theta})\left ( \frac{1}{2(1-\cos\theta)}+ \sum_{\ell=1}^\infty\frac{\gamma^{\ell}(1+\gamma^{2\ell})}{1+\gamma^{4\ell}-2\gamma^{2\ell}\cos\theta} \right ),
\end{eqnarray*}
hence
\begin{eqnarray*}
|F_\gamma (\zeta)|&=&\frac{1}{2|\sin(\theta/2)|}+2|\sin(\theta/2)| \sum_{\ell=1}^\infty\frac{\gamma^{\ell}(1+\gamma^{2\ell})}{1+\gamma^{4\ell}-2\gamma^{2\ell}\cos\theta}\\
&=&|\sin(\theta/2)|\sum_{\ell \in \Z}\frac{\gamma^{\ell}(1+\gamma^{2\ell})}{1+\gamma^{4\ell}-2\gamma^{2\ell}\cos\theta}\ .
\end{eqnarray*}
Similarly, if $\zeta=\gamma {\rm e}^{i\phe}$, we have
\begin{eqnarray*}
F_\gamma (\zeta)&=&\sum_{k=0}^\infty \frac{\gamma ^k}{1-\gamma ^{2k+1}{\rm e}^{i\phe}}+\sum_{\ell =0}^\infty \frac{\gamma ^{-\ell-1}}{1-\gamma ^{-2\ell -1}{\rm e}^{i\phe}}\\
&=&\sum_{k=0}^\infty \frac{\gamma ^k}{1-\gamma ^{2k+1}{\rm e}^{i\phe}}+\sum_{\ell=0}^\infty \frac{\gamma ^{\ell }}{\gamma ^{2\ell +1}-{\rm e}^{i\phe}}\\
&=&\sum_{\ell=0}^\infty\frac{\gamma^{\ell}(1+\gamma^{2\ell+1})(1-{\rm e}^{-i\phe})}{1+\gamma^{4\ell+2}-2\gamma^{2\ell+1}\cos\phe},
\end{eqnarray*}
so that
\begin{eqnarray*}
|F_\gamma (\zeta)|&=&2|\sin(\phe/2)| \sum_{\ell=0}^\infty\frac{\gamma^{\ell}(1+\gamma^{2\ell+1})}{1+\gamma^{4\ell+2}-2\gamma^{2\ell+1}\cos\phe}\\
&=&|\sin(\phe/2)|\sum_{\ell \in \Z}\frac{\gamma^{\ell}(1+\gamma^{2\ell+1})}{1+\gamma^{4\ell+2}-2\gamma^{2\ell+1}\cos\phe}\ .
\end{eqnarray*}
Consequently,
\begin{eqnarray*}
\min_{|\zeta|=1}|F_\gamma (\zeta)|-\gamma^{1/2}\max _{|\zeta|=\gamma}|F_\gamma (\zeta)|&=&\min_{\theta \in \T}|\sin(\theta/2)|\sum_{\ell \in \Z}\frac{\gamma^{\ell}(1+\gamma^{2\ell})}{1+\gamma^{4\ell}-2\gamma^{2\ell}\cos\theta}\\
&&-\max_{\phe \in \T}|\sin(\phe/2)|\sum_{\ell \in \Z}\frac{\gamma^{\ell+1/2}(1+\gamma^{2\ell+1})}{1+\gamma^{4\ell+2}-2\gamma^{2\ell+1}\cos\phe}
\end{eqnarray*}
Set, for $x\in \R, \theta \in \T\setminus \{0\}$,
$$f_{\gamma ,\theta} (x)=|\sin(\theta/2)|\frac{\gamma ^x(1+\gamma^{2x})}{1+\gamma^{4x}-2\gamma^{2x}\cos\theta}\ .$$
Then we are reduced to proving that
$$\inf_{\theta \in \T\setminus \{0\}}\sum_{k\in \Z}f_{\gamma ,\theta}(k)-\sup_{\phe \in \T\setminus \{0\} }\sum_{k\in \Z}f_{\gamma ,\phe}\left (k+\frac 12\right )>0\ .$$
Applying the Poisson summation formula, we have
$$\sum_{k\in \Z}f_{\gamma ,\theta}(k)=\sum_{n\in \Z}\hat f_{\gamma,\theta }(2\pi n)\ ,\ \sum_{k\in \Z}f_{\gamma ,\phe}\left (k+\frac 12\right )=\sum_{n\in \Z}(-1)^n\hat f_{\gamma,\phe }(2\pi n),$$
where 
\begin{eqnarray*}
\hat f_{\gamma, \theta }(\xi )&=&|\sin(\theta/2)|\int_\R \frac{\gamma ^x(1+\gamma^{2x})}{1+\gamma^{4x}-2\gamma^{2x}\cos\theta}{\rm e}^{-ix\xi}\, dx\\
&=&\frac{|\sin(\theta/2)|}{|\log \gamma |}\int_0^\infty \frac{(1+t^2)t^{-i\xi /\log \gamma}}{1+t^4-2t^2\cos\theta }\, dt\\
&=&\frac{|\sin(\theta/2)|}{2|\log \gamma |}\int_0^\infty \frac{(1+y)y^{-i\xi /2\log \gamma-1/2}}{1+y^2-2y\cos\theta }\, dy\ ,
\end{eqnarray*}
where we have set $t=\gamma ^x\ ,\ y=t^2$. We calculate the above integral by introducing the holomorphic function
$$g(z)=\frac{|\sin(\theta/2)|}{2|\log \gamma |}\frac{(1+z)z^{-i\xi /2\log \gamma-1/2}}{1+z^2-2z\cos\theta},$$
on the domain $\C\setminus \R_+$, where the argument of $z$ belongs to $(0,2\pi)$. Integrating on the contour
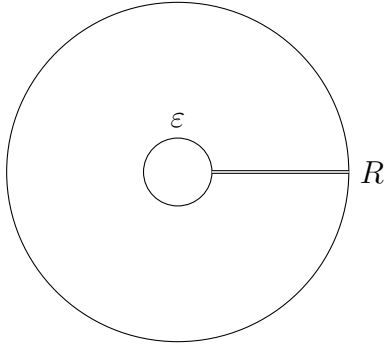
\begin{center}
\begin{tikzpicture}[scale=1.5]
\draw (0,0) circle (1.5);
\draw (0,0) circle (0.3);
\draw [double] (0.3,0)--(1.5,0);
\draw (1.5,0) node[right]{$R$};
\draw (0,0.3) node[above]{$\e$};
\end{tikzpicture}
\end{center}
and making $R\to \infty ,\e \to 0$, we obtain, by the residue theorem, assuming $\theta \in (0,2\pi)$ with no loss of generality,
\begin{eqnarray*}
\hat f_{\gamma, \theta }(\xi )\left (1+{\rm e}^{\pi \xi/\log\gamma}\right )&=&2i\pi \left [{\rm Res}\left (g(z),z={\rm e}^{i\theta}\right )+ {\rm Res}\left (g(z),z={\rm e}^{-i\theta}\right )\right ]\\
&=&\frac{i\pi \sin(\theta/2)}{|\log \gamma|}\left (\frac{ 2\cos (\theta/2)  }{2i\sin\theta }{\rm e}^{\theta \xi/2\log \gamma}+\frac{ 2\cos (\theta/2)  }{2i\sin\theta }{\rm e}^{(2\pi -\theta )\xi/2\log \gamma}\right )\\
&=&\frac{\pi}{2|\log \gamma|}\left ( {\rm e}^{\theta \xi/2\log \gamma}+  {\rm e}^{(2\pi -\theta )\xi/2\log \gamma}\right )\ .
\end{eqnarray*}
We infer
$$\hat f_{\gamma, \theta }(\xi )=\frac{\pi}{2|\log \gamma|}\frac{\cosh\left (\frac{(\pi -\theta)\xi}{2\log \gamma}\right )}{\cosh\left (\frac{\pi \xi}{2\log \gamma}\right )}\ ,\ \theta \in (0,2\pi)\ .$$
\s
Finally, for $\theta,\phe \in (0,2\pi)$,
\begin{eqnarray*}
&&\sum_{k\in \Z}f_{\gamma ,\theta}(k)-\sum_{k\in \Z}f_{\gamma ,\phe}\left (k+\frac 12\right )\\
&&=\frac{\pi}{2|\log \gamma|}\left (\sum_{n\in \Z}\frac{\cosh\left (\frac{(\pi -\theta)\pi n}{\log \gamma}\right )}{\cosh\left (\frac{\pi ^2n}{\log \gamma}\right )}-\sum_{n\in \Z}(-1)^n\frac{\cosh\left (\frac{(\pi -\phe)\pi n}{\log \gamma}\right )}{\cosh\left (\frac{\pi ^2n}{\log \gamma}\right )}\right )\\
&&=\frac{\pi}{|\log \gamma|}\left (\sum_{n=1}^\infty \frac{\cosh\left (\frac{(\pi -\theta)\pi n}{\log \gamma}\right )}{\cosh\left (\frac{\pi ^2n}{\log \gamma}\right )} +\sum_{n=1}^\infty (-1)^{n+1}\frac{\cosh\left (\frac{(\pi -\phe)\pi n}{\log \gamma}\right )}{\cosh\left (\frac{\pi ^2n}{\log \gamma}\right )}\right )\\
&&\geq \frac{\pi}{|\log \gamma|}\sum_{n=1}^\infty \frac{1}{\cosh\left (\frac{\pi ^2n}{\log \gamma}\right )}
\end{eqnarray*}
since the second series is an alternating series of the form
$$\sum_{n=1}^\infty  (-1)^{n+1} a_n\ ,$$
with $a_n$ decaying to $0$ as $n\to \infty$.
\s
Therefore
$$\min_{|\zeta|=1}|F_\gamma (\zeta)|-\gamma^{1/2}\max _{|\zeta|=\gamma}|F_\gamma (\zeta)|\geq \frac{\pi}{|\log \gamma|}\sum_{n=1}^\infty \frac{1}{\cosh\left (\frac{\pi ^2n}{\log \gamma}\right )}>0\ .$$
The proof is complete.
\end{proof}

Lemma \ref{NonVanishing} implies that $\Phi$ has no zeroes for $|\zeta|=1$ and $|z|\le 1$. By continuity, it has no zeroes in a neighbourhood of this set. Hence Lemma \ref{ZeroesPhi} is proved.

\subsection{Studying the index}

Let us first recall the definition of the index.
For $0<R<\infty$, we denote by $\mathcal C_R$ the circle $$\{z\in\C,\;|z|=R\}.$$  Let $f$ be a holomorphic function near $\mathcal C_R$, with no zero on $\mathcal C_R$. The index on $\mathcal C_R$ around $0$ of  $f$   is given by 
$${\rm Ind}_{f(\mathcal C_R)}(0):=\frac 1{2i\pi}\int_{\mathcal C_R}\frac{f'(\zeta)}{f(\zeta)}d\zeta.$$
 In this section, we prove the following lemma.
\begin{lemma}\label{index}
For any $r<1$ sufficiently close to $1$, the function $$\zeta\mapsto F_\gamma (r\zeta)$$ has index zero on the unit circle.
\end{lemma}
Notice that $\Phi (0,r\zeta )=F_\gamma (r\zeta )$. As the index is valued in $\Z$ and the map $z\mapsto \Phi(z,\zeta)$ is smooth, Lemma \ref{index} implies that the index of $\zeta\mapsto \Phi(z,r\zeta)$ is zero for any $z$ with $|z|\leq 1+\rho $ as long as $r$ is sufficiently close to $1$.
\begin{corollary}
For any $r<1$ sufficiently close to $1$, the function $$\zeta\mapsto \Phi(z,r\zeta)$$ has index zero for any $z$.
\end{corollary}
This corollary will complete the proof of Proposition \ref{invertibility}.
\begin{proof} Proof of Lemma \ref{index}.\\
{We could  use Remark \ref{elliptic} in order to reduce to properties of Weierstrass $\mathfrak P$ function. However, for the convenience of the reader, we prefer to give a self--contained proof}. Let us assume that $R$ is chosen so that $R\neq \gamma^{2\ell},\; \ell \in \Z$ and $F_\gamma\neq 0$ on $\mathcal C_R$. We consider the index of $F_\gamma$ on $\mathcal C_R$ around $0$
$$I(R):={\rm Ind}_{F_\gamma(\mathcal C_R)}(0):=\frac 1{2i\pi}\int_{\mathcal C_R}\frac{F'_\gamma(\zeta)}{F_\gamma(\zeta)}d\zeta.$$
Statement of Lemma \ref{index} is equivalent to $$I(1^-):=\lim_{R\to 1^-} I(R)=0.$$ By definition, $I$ is valued in $\Z$ and is continuous on the intervals corresponding to the circles avoiding the zeroes and the poles of $F_\gamma$.
From properties (\ref{properties}), one has
\begin{equation}\label{general}
I(R)+I\left(\frac 1R\right)=-1,\; I(R\gamma^2)=I(R).
\end{equation}
In particular \begin{equation}\label{I(rho)}
I(R)+I\left(\frac {\gamma^2}R\right)=-1
\end{equation}
and
\begin{equation}\label{firstEquality}
I(1^+)=I((\gamma^2)^+) 
\end{equation}where $I(r^\pm)=\lim_{t\to r^\pm} I(t)$. We are going to compute $I((\gamma^2)^+)$ by another way, using the zeroes and the poles of $F_\gamma$. \\
 Let us first collect some basic relations.\\
Let $n$ be the number of zeroes in the annulus $$\{z\in\C,\; \gamma<|z|<1\}.$$ Since there is no poles inside this annulus, one has 
\begin{equation}\label{first}
n=I(1^-)-I(\gamma^+)
\end{equation}  From Equation (\ref{I(rho)}) with $R=1^-$ and $R=\gamma^+$, 
\begin{eqnarray*}
I(1^-)+I((\gamma^2)^+)&=&-1\; \mbox{ and }\\
I(\gamma^+)+I(\gamma^-)&=&-1.
\end{eqnarray*}
Substracting these equalities gives $I(1^-)-I(\gamma^+)= I(\gamma^-)-I((\gamma^2)^+)$ hence 
 \begin{equation}\label{second}
n=I(\gamma^-)-I((\gamma^2)^+).
\end{equation} 
Denote by $m$ the number of zeroes on $\mathcal C_\gamma$. As $\gamma$ is a zero of $F_\gamma$, {$m\geq1$}, and
\begin{equation}\label{trois}
m=I(\gamma^+)-I(\gamma^-)
\end{equation} 
since there is no pole on $\mathcal C_\gamma$.
Denote by $N$ the number of zeroes on $\mathcal C_1$ then
\begin{equation}\label{fourth}I(1^+)-I(1^-)=N-1
\end{equation}\label{third} since $1$ is the only pole on  $\mathcal C_1$. 

Now, we compute $I((\gamma^2)^+)$. 
From equality (\ref{second}),
\begin{eqnarray*}
I((\gamma^2)^+)&=&I(\gamma^-)-n\\
&=&I(\gamma^+)-m-n\mbox{ from equality }(\ref{trois})\\
&=&I(1^-)-n-m-n\mbox{ from equality }(\ref{first})\\
&=&I(1^+)-(N-1)-m-2n\mbox{ from equality }(\ref{fourth})
\end{eqnarray*}

Recalling equality (\ref{firstEquality}), we conclude that $N+2n+m=1$ so $n=0$ and $N+m=1$. Since $m\geq 1$, this implies $N=0$ and $m=1$.
From Equation (\ref{fourth}) and the equality  $I(1^+)+I(1^-)=-1$ (Equation (\ref{I(rho)}) with $R=1^+$), one concludes 
$I(1^+)=-1$ and $I(1^-)=0$ as required.
\end{proof}

\begin{appendix}
\section{A formula for the $C^1$ norm }\label{C1}
Let $u\in L^2_+(\T)$ be a rational function  corresponding to the finite list of singular values
$\rho_1>\sigma_1>\dots >\rho_N>\sigma_N$
and angles $\psi_r=0$ for $r=1,\dots ,2N$. Then we  checked in \cite{GG2}, \cite{GG3} that this cancellation of the angles   precisely corresponds to the positivity of the operators $\Gamma_{\hat u}$ and $\tilde \Gamma_{\hat u}$ on $\ell^2(\Z_+)$.
The representation formula \eqref{uCN}, \eqref{CN} then reduces to 
$$u(z)=\langle {\mathscr C}_N(z)^{-1}({\bf 1}_N),{\bf 1}_N\rangle \ ,$$
with
\begin{equation}\label{rep}
{\mathscr C}_N(z):=\left (\frac{s_{2j-1}-s_{2k}z}{s_{2j-1}^2-s_{2k}^2} \right )_{1\leq j,k\leq N}\ ,
\end{equation}
Furthermore,  the positivity of the Hankel matrices $\Gamma_{\hat u}$ and $\tilde \Gamma_{\hat u}$ implies the positivity of the Fourier coefficitients of $u$, since, denoting by $(e_n)_{n\ge 0}$ the canonical basis of $\ell ^2(\Z_+)$,
$$\langle \Gamma_{\hat u}e_n,e_n\rangle =\hat u(2n)\ ,\ \langle \tilde \Gamma_{\hat u}e_n,e_n\rangle =\hat u(2n+1)\ .$$
Therefore the $C^1$ norm of $u$ on $\T$ is given by
$$S(u):= \sum_{n=1}^\infty n\hat u(n)\ .$$
The  lemma below explicitly computes $S(u)$.
\begin{lemma}\label{S}
$$S(u)=\sum_{k=1}^N\sigma_k\left (\prod_{j=1}^N \frac{\rho_j+\sigma_k}{\rho_j-\sigma_k}\right )\left (\prod_{\ell \ne k}\frac{\sigma_k+\sigma_\ell}{\sigma_\ell -\sigma_k}\right )\ ,$$
where every term in the above sum is positive.
\end{lemma}
\begin{proof}
We have
$$S(u)=u'(1)=\langle \dot {\mathscr C}\mathscr C(1)^{-1}({\bf 1}), ^t\mathscr C(1)^{-1}({\bf 1})\rangle \ ,$$
with
$$\mathscr C(1):=\left (\frac{1}{\rho_j+\sigma_k}\right )_{1\leq j,k\leq N}\ ,\ \dot{\mathscr C}:=\left (\frac{\sigma_k}{\rho_j^2-\sigma_k^2}\right )_{1\leq j,k\leq N}\ .$$
Notice that $\mathscr C(1)$ is a Cauchy matrix,  so that the expression of $\mathscr C(1)^{-1}({\bf 1})$ is explicit. We have
\begin{equation}\label{Cauchy}
\mathscr C(1)^{-1}({\bf 1})=\left ( \frac{\prod_{j=1}^N(\rho_j+\sigma_k)}{\prod_{\ell \ne k}(\sigma_k-\sigma_\ell) }\right )_{1\le k\le N}\ .\end{equation}
Let us give a simple proof of this formula, inspired from calculations in \cite{GP2}. Denote by $x_k, k=1,\dots ,N$, the components of $\mathscr C(1)^{-1}({\bf 1})$. We have 
$$\sum_{k=1}^N \frac{x_k}{\rho_j+\sigma_k}=1\ ,\ j=1,\dots ,N\ .$$
Consider the polynomial functions
$$Q(\rho ):=\prod_{k=1}^N(\rho +\sigma_k)\ ,\ P(\rho ):=Q(\rho )\sum_{k=1}^N \frac{x_k}{\rho+\sigma_k}\ .$$
Then $Q$ has degree $N$, $P$ has degree at most $N-1$ and 
$$P(\rho_j)=Q(\rho_j)\ ,\ j=1,\dots N\ .$$
Since $Q-P$ is a unitary polynomial of degree $N$ which cancels at $\rho_j, j=1,\dots ,N$, we have
$$Q(\rho )-P(\rho )=\prod_{j=1}^N (\rho -\rho_j)\ .$$
Consequently, 
$$P(-\sigma_k)=-\prod_{j=1}^N (-\sigma_k-\rho_j)=(-1)^{N-1}\prod_{j=1}^N (\sigma_k+\rho_j)\ .$$
Since 
$$x_k=\frac{P(-\sigma_k)}{Q'(-\sigma_k)}\ ,$$
this yields \eqref{Cauchy}. Similarly, we have
\begin{equation}\label{CauchyT}
\ ^t\mathscr C(1)^{-1}({\bf 1})=\left (\frac{\prod_{\ell =1}^N(\rho_j+\sigma_\ell )}{\prod_{i \ne j}(\rho_j-\rho_i) }\right )_{1\le j\le N}\ .
\end{equation}
Coming back to the proof of Lemma \ref{S}, we have, in view of \eqref{Cauchy} and \eqref{CauchyT}, 
$$S(u)=\sum_{j,k=1}^N\mu_{jk}^{(N)}\ ,\ \mu_{jk}^{(N)}:=\sigma_k\frac{\rho_j+\sigma_k}{\rho_j-\sigma_k}\left (\prod_{i\ne j}\frac{\rho_i+\sigma_k}{\rho_j-\rho_i}\right )\left (\prod_{\ell \ne k}\frac{\rho_j+\sigma_\ell}{\sigma_k-\sigma_\ell}\right )\ .$$
Multiplying and dividing $\mu_{jk}^{(N)}$ by $\prod_{i\ne j}(\rho_i-\sigma_k)$, we have, for every $k$,
$$\sum_{j=1}^N \mu_{jk}^{(N)}=\frac{\sigma_k R(\sigma_k)}{\prod_{\ell \ne k}(\sigma_k-\sigma_\ell)}\prod_{i=1}^N\frac{\rho_i+\sigma_k}{\rho_i-\sigma_k}\ ,$$
with
$$R(\sigma )=\sum_{j=1}^N \prod_{i\ne j}\frac{\rho_i-\sigma}{\rho_j-\rho_i}\prod_{\ell \ne k}(\rho_j+\sigma_\ell)\ .$$
Notice that, for every $j=1,\dots ,N$, 
$$R(\rho_j)=(-1)^{N-1}\prod_{\ell \ne k}(\rho_j+\sigma_\ell)\ .$$
Since $R$ has degree $N-1$, we infer
$$R(\sigma )=(-1)^{N-1}\prod_{\ell \ne k}(\sigma +\sigma_\ell )\ ,$$
so that
$$\sum_{j=1}^N \mu_{jk}^{(N)}=\frac{\sigma_k (-1)^{N-1}\prod_{\ell \ne k}(\sigma_k +\sigma_\ell )}{\prod_{\ell \ne k}(\sigma_k-\sigma_\ell)}\prod_{i=1}^N\frac{\rho_i+\sigma_k}{\rho_i-\sigma_k}\ ,$$
which is the claimed formula. The positivity of each term is an easy consequence of the inequalities $\rho_1>\sigma_1>\rho_2>\sigma_2>\dots $
\end{proof}
As a consequence of Lemma \ref{S}, we retain the following  inequality, obtained after discarding most of the factors bigger than 1 in each of the products.
\begin{corollary}\label{ineqC1}
$$\Vert u\Vert_{C^1}\geq \sum_{k=1}^N\frac{\sigma_k(\rho_k+\sigma_k)}{\rho_k-\sigma_k}\ .$$
\end{corollary}
Notice that this implies inequality \eqref{estC1}. Unfortunately, at this stage we do not have arguments allowing to extend this inequality to non rational functions $u$, which would imply that $u\notin C^1$ if 
$$\sum_{k=1}^\infty\frac{\rho_k\sigma_k}{\rho_k-\sigma_k}=\infty \ .$$
\end{appendix}

\end{document}